\documentclass{amsart}
\usepackage{amsmath,amssymb,bm}
\usepackage{amsmath,amsthm,amssymb,cases}
\usepackage[all]{xy}
\usepackage{amsmath}
\usepackage{amssymb}
\usepackage{amscd}
\usepackage{amsthm}
\usepackage[dvips]{graphicx}
\usepackage{color}

\theoremstyle{definition}
\newtheorem{defn}{Definition}[section]

\newtheorem{rem}[defn]{Remark}
\theoremstyle{plain}
\newtheorem{thm}[defn]{Theorem}
\newtheorem{prop}[defn]{Proposition}
\newtheorem{lem}[defn]{Lemma}
\newtheorem{cor}[defn]{Corollary}

\numberwithin{equation}{section}
\allowdisplaybreaks
\title[generating level 2 twist subgroup]{A finite generating set for the level 2 twist subgroup of the mapping class group of a closed non-orientable surface}
\author[R.~Kobayashi]{Ryoma Kobayashi}
\address{
(Ryoma Kobayashi)
Department of General Education, Ishikawa National College of Technology, Tsubata, Ishikawa, 929-0392, Japan 
}
\email{kobayashi\_ryoma@ishikawa-nct.ac.jp}
\author[G.~Omori]{Genki Omori}
\address{
(Genki Omori)
Department of Mathematics,
Tokyo Institute of Technology,
Oh-okayama, Meguro, Tokyo 152-8551, Japan
}

\email{omori.g.aa@m.titech.ac.jp}

\subjclass[2010]{57M05, 57M07, 57M20, 57M60}

\date{\today}
\begin{document}
\maketitle
\begin{abstract}
We obtain a finite generating set for the level 2 twist subgroup of the mapping class group of a closed non-orientable surface. The generating set consists of crosscap pushing maps along non-separating two-sided simple loops and squares of Dehn twists along non-separating two-sided simple closed curves. We also prove that the level 2 twist subgroup is normally generated in the mapping class group by a crosscap pushing map along a non-separating two-sided simple loop for genus $g\geq 5$ and $g=3$. As an application, we calculate the first homology group of the level 2 twist subgroup for genus $g\geq 5$ and $g=3$.
\end{abstract}

\section{Introduction}

Let $N_{g,n}$ be a compact connected non-orientable surface of genus $g\geq 1$ with $n\geq 0$ boundary components. The surface $N_g=N_{g,0}$ is a connected sum of $g$ real projective planes. The {\it mapping class group} $\mathcal{M}(N_{g,n})$ of $N_{g,n}$ is the group of isotopy classes of self-diffeomorphisms on $N_{g,n}$ fixing the boundary pointwise and the {\it twist subgroup} $\mathcal{T}(N_{g,n})$ of $\mathcal{M}(N_{g,n})$ is the subgroup of $\mathcal{M}(N_{g,n})$ generated by all Dehn twists along two-sided simple closed curves. Lickorish~\cite{Lickorish2} proved that $\mathcal{T}(N_g)$ is an index 2 subgroup of $\mathcal{M}(N_g)$ and the non-trivial element of $\mathcal{M}(N_g)/\mathcal{T}(N_g)\cong \mathbb Z/2\mathbb Z=:\mathbb Z_2$ is represented by a ``Y-homeomorphism''. We define a Y-homeomorphism in Section~\ref{Preliminaries}. Chillingworth~\cite{Chillingworth} gave an explicit finite generating set for $\mathcal{T}(N_g)$ and showed that $\mathcal{T}(N_2)\cong \mathbb Z_2$. The first homology group $H_1(G)$ of a group $G$ is isomorphic to the abelianization $G^{ab}$ of $G$. The group $H_1(\mathcal{T}(N_g))$ is trivial for $g\geq 7$, $H_1(\mathcal{T}(N_3))\cong \mathbb Z_{12}$, $H_1(\mathcal{T}(N_4))\cong \mathbb Z_2\oplus \mathbb Z$ and $H_1(\mathcal{T}(N_g))\cong \mathbb Z_2$ for $g=5, 6$. These results were shown by Korkmaz~\cite{Korkmaz1} for $g\geq 7$ and by Stukow~\cite{Stukow} for the other cases. 

Let $\Sigma _{g,n}$ be a compact connected orientable surface of genus $g\geq 0$ with $n\geq 0$ boundary components. The mapping class group $\mathcal{M}(\Sigma _{g,n})$ of $\Sigma _{g,n}$ is the group of isotopy classes of orientation preserving self-diffeomorphisms on $\Sigma _{g,n}$ fixing the boundary pointwise. Let $S$ be either $N_{g,n}$ or $\Sigma _{g,n}$. For $n=0$ or $1$, we denote by $\Gamma _2(S)$ the subgroup of $\mathcal{M}(S)$ which consists of elements acting trivially on $H_1(S;\mathbb Z_2)$. $\Gamma _2(S)$ is called the {\it level 2 mapping class group} of $S$. For a group $G$, a normal subgroup $H$ of $G$ and a subset $X$ of $H$, $H$ is {\it normally generated in $G$ by $X$} if $H$ is the normal closure of $X$ in $G$. In particular, for $X=\{x_1,\dots ,x_n\}$, if $H$ is the normal closure of $X$ in $G$, we also say that $H$ is {\it normally generated in $G$ by $x_1,\dots ,x_n$}. In the case of orientable surfaces, Humphries~\cite{Humphries} proved that $\Gamma _2(\Sigma _{g,n})$ is normally generated in $\mathcal{M}(\Sigma _{g,n})$ by the square of the Dehn twist along a non-separating simple closed curve for $g\geq 1$ and $n=0$ or $1$. In the case of non-orientable surfaces, Szepietowski~\cite{Szepietowski1} proved that $\Gamma _2(N_g)$ is normally generated in $\mathcal{M}(N_g)$ by a Y-homeomorphism for $g\geq 2$. Szepietowski~\cite{Szepietowski2} also gave an explicit finite generating set for $\Gamma _2(N_g)$. This generating set is minimal for $g=3$, $4$. Hirose and Sato~\cite{Hirose-Sato} gave a minimal generating set for $\Gamma _2(N_g)$ when $g\geq 5$ and showed that $H_1(\Gamma _2(N_g))\cong \mathbb Z_2^{\binom{g}{3}+\binom{g}{2}}$.

We denote by $\mathcal{T}_2(N_g)$ the subgroup of $\mathcal{T}(N_g)$ which consists of elements acting trivially on $H_1(N_g;\mathbb Z_2)$ and we call $\mathcal{T}_2(N_g)$ the {\it level 2 twist subgroup of} $\mathcal{M}(N_g)$. Recall that $\mathcal{T}(N_2)\cong \mathbb Z_2$ and Chillingworth~\cite{Chillingworth} proved that $\mathcal{T}(N_2)$ is generated by the Dehn twist along a non-separating two-sided simple closed curve. $\mathcal{T}_2(N_2)$ is a trivial group because Dehn twists along non-separating two-sided simple closed curves induce nontrivial actions on $H_1(N_g;\mathbb Z_2)$. Let $\operatorname{Aut}(H_1(N_g;\mathbb Z_2),\cdot )$ be the group of automorphisms on $H_1(N_g;\mathbb Z_2)$ preserving the intersection form $\cdot $ on $H_1(N_g;\mathbb Z_2)$. Since the action of $\mathcal{M}(N_g)$ on $H_1(N_g;\mathbb Z_2)$ preserves the intersection form $\cdot $, there is the natural homomorphism from $\mathcal{M}(N_g)$ to $\operatorname{Aut}(H_1(N_g;\mathbb Z_2),\cdot )$. McCarthy and Pinkall~\cite{McCarthy-Pinkall} showed that the restriction of the homomorphism to $\mathcal{T}(N_g)$ is surjective. Thus $\mathcal{T}_2(N_g)$ is finitely generated. 

In this paper, we give an explicit finite generating set for $\mathcal{T}_2(N_g)$ (Theorem~\ref{main-thm}). The generating set consists of ``crosscap pushing maps'' along non-separating two-sided simple loops and squares of Dehn twists along non-separating two-sided simple closed curves. We review the crosscap pushing map in Section~\ref{Preliminaries}. We can see the generating set for $\mathcal{T}_2(N_g)$ in Theorem~\ref{main-thm} is minimal for $g=3$ by Theorem~\ref{first-homology}. We prove Theorem~\ref{main-thm} in Section~\ref{section-finite-gen}. In the last part of Subsection~\ref{proof_main}, we also give the smaller finite generating set for $\mathcal{T}_2(N_g)$ (Theorem~\ref{main-thm2}). However, the generating set consists of crosscap pushing maps along non-separating two-sided simple loops, squares of Dehn twists along non-separating two-sided simple closed curves and squares of Y-homeomorphisms.

By using the finite generating set for $\mathcal{T}_2(N_g)$ in Theorem~\ref{main-thm}, we prove the following theorem in Section~\ref{section-normal-gen}.

\begin{thm}\label{normal-gen}
For $g=3$ and $g\geq 5$, $\mathcal{T}_2(N_g)$ is normally generated in $\mathcal{M}(N_g)$ by a crosscap pushing map along a non-separating two-sided simple loop (See Figure~\ref{crosscap_normal1}). 

$\mathcal{T}_2(N_4)$ is normally generated in $\mathcal{M}(N_4)$ by a crosscap pushing map along a non-separating two-sided simple loop and the square of the Dehn twist along a non-separating two-sided simple closed curve whose complement is a connected orientable surface (See Figure~\ref{gamma1234}).
\end{thm}
The x-marks as in Figure~\ref{crosscap_normal1} and Figure~\ref{gamma1234} mean M\"{o}bius bands attached to boundary components in this paper and we call the M\"{o}bius band the {\it crosscap}. The group which is normally generated in $\mathcal{M}(N_g)$ by the square of the Dehn twist along a non-separating two-sided simple closed curve is a subgroup of $\mathcal{T}_2(N_g)$ clearly. The authors do not know whether $\mathcal{T}_2(N_g)$ is generated by squares of Dehn twists along non-separating two-sided simple closed curves or not.

\begin{figure}[h]
\includegraphics[scale=0.70]{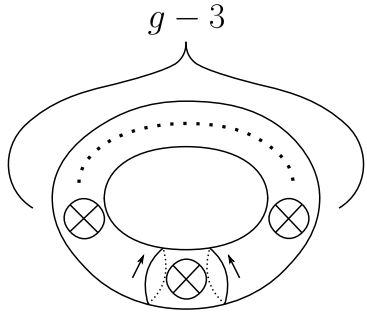}
\caption{A crosscap pushing map along a non-separating two-sided simple loop is described by a product of Dehn twists along non-separating two-sided simple closed curves as in the figure.}\label{crosscap_normal1}
\end{figure}

\begin{figure}[h]
\includegraphics[scale=0.70]{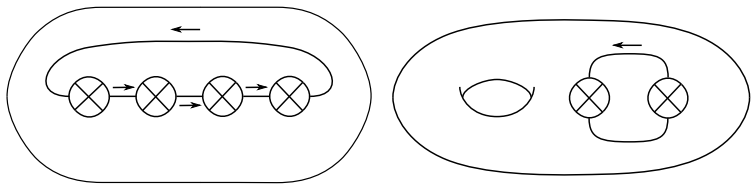}
\caption{A non-separating two-sided simple closed curve on $N_4$ whose complement is a connected orientable surface.}\label{gamma1234}
\end{figure}

As an application of Theorem~\ref{normal-gen}, we calculate $H_1(\mathcal{T}_2(N_g))$ for $g\geq 5$ in Section~\ref{section-first-homology} and we obtain the following theorem.

\begin{thm}\label{first-homology} 
For $g=3$ and $g\geq 5$, the first homology group of $\mathcal{T}_2(N_g)$ is as follows:
\[
H_1(\mathcal{T}_2(N_g))\cong \left\{ \begin{array}{ll}
 \mathbb Z^2\oplus \mathbb Z_2&\text{if} \ g=3,   \\
 \mathbb Z_2^{\binom{g}{3}+\binom{g}{2}-1}&\text{if} \ g\geq 5.
 \end{array} \right.
\]
\end{thm}
In this proof, we use the five term exact sequence for an extension of a group for $g\geq 5$. The authors do not know the first homology group of $\mathcal{T}_2(N_4)$.

\section{Preliminaries}\label{Preliminaries}

\subsection{Crosscap pushing map}
Let $S$ be a compact surface and let $e:D^\prime \hookrightarrow {\rm int}S$ be a smooth embedding of the unit disk $D^\prime \subset \mathbb C$. Put $D:=e(D^\prime )$. Let $S^\prime $ be the surface obtained from $S-{\rm int}D$ by the identification of antipodal points of $\partial D$. We call the manipulation that gives $S^\prime $ from $S$ the {\it blowup of} $S$ {\it on} $D$. Note that the image $M$ of the regular neighborhood of $\partial D$ in $S-{\rm int}D$ by the blowup of $S$ on $D$ is a crosscap, where a crosscap is a M\"{o}bius band in the interior of a surface. Conversely, the {\it blowdown of} $S^\prime$ {\it on }$M$ is the following manipulation that gives $S$ from $S^\prime $. We paste a disk on the boundary obtained by cutting $S$ along the center line $\mu $ of $M$. The blowdown of $S^\prime $ on $M$ is the inverse manipulation of the blowup of $S$ on $D$.

Let $x_0$ be a point of $N_{g-1}$ and let $e:D^\prime \hookrightarrow N_{g-1}$ be a smooth embedding of a unit disk $D^\prime \subset \mathbb C$ to $N_{g-1}$ such that the interior of $D:=e(D^\prime )$ contains $x_0$. Let $\mathcal{M}(N_{g-1},x_0)$ be the group of isotopy classes of self-diffeomorphisms on $N_{g-1}$ fixing the point $x_0$, where isotopies also fix $x_0$. Then we have the {\it blowup homomorphism} 
\[
\varphi :\mathcal{M}(N_{g-1},x_0)\rightarrow \mathcal{M}(N_g)
\]
that is defined as follows. For $h \in \mathcal{M}(N_{g-1},x_0)$, we take a representative $h^\prime $ of $h$ which satisfies either of the following conditions: (a) $h^\prime |_{D}$ is the identity map on $D$, (b) $h^\prime (x)=e(\overline{e^{-1}(x)})$ for $x\in D$. Such $h^\prime $ is compatible with the blowup of $N_{g-1}$ on $D$, thus $\varphi (h)\in \mathcal{M}(N_g)$ is induced and well defined (c.f. \cite[Subsection~2.3]{Szepietowski1}). 

The {\it point pushing map} 
\[
j:\pi _1(N_{g-1},x_0)\rightarrow \mathcal{M}(N_{g-1},x_0)
\]
is a homomorphism that is defined as follows. For $\gamma \in \pi _1(N_{g-1},x_0)$, $j(\gamma )\in \mathcal{M}(N_{g-1},x_0)$ is described as the result of pushing the point $x_0$ once along $\gamma $. Note that for $x$, $y\in \pi _1(N_{g-1})$, $yx$ means $yx(t)=x(2t)$ for $0\leq t\leq \frac{1}{2}$ and $yx(t)=y(2t-1)$ for $\frac{1}{2}\leq t\leq 1$, and for elements $[f]$, $[g]$ of the mapping class group, $[f][g]$ means $[f\circ g]$.

We define the {\it crosscap pushing map} as the composition of homomorphisms:
\[
\psi :=\varphi \circ j:\pi _1(N_{g-1},x_0)\rightarrow \mathcal{M}(N_g).
\]
For $\gamma \in \pi _1(N_{g-1},x_0)$, we also call $\psi (\gamma )$ the {\it crosscap pushing map along $\gamma $}. 
Remark that for $\gamma $, $\gamma ^\prime \in \pi _1(N_{g-1},x_0)$, $\psi (\gamma )\psi (\gamma ^\prime )=\psi (\gamma \gamma ^\prime )$. The next two lemmas follow from the description of the point pushing map (See \cite[Lemma~2.2, Lemma~2.3]{Korkmaz2}). 

\begin{lem}\label{pushing1}
For a two-sided simple loop $\gamma $ on $N_{g-1}$ based at $x_0$, suppose that $\gamma _1$, $\gamma _2$ are two-sided simple closed curves on $N_{g-1}$ such that $\gamma _1\sqcup \gamma _2$ is the boundary of the regular neighborhood $N$ of $\gamma$ in $N_{g-1}$ whose interior contains $D$. Then for some orientation of $N$, we have
\[
\psi (\gamma )=\varphi (t_{\gamma _1}t_{\gamma _2}^{-1})=t_{\widetilde {\gamma _1}}t_{\widetilde {\gamma _2}}^{-1},
\] 
where $\widetilde {\gamma _1}$, $\widetilde {\gamma _2}$ are images of $\gamma _1$, $\gamma _2$ to $N_g$ by blowups respectively (See Figure~\ref{crosscap_def_twist}). 
\end{lem}
Let $\mu $ be a one-sided simple closed curve and let $\alpha $ be a two-sided simple closed curve on $N_g$ such that $\mu $ and $\alpha $ intersect transversely at one point. For these simple closed curves $\mu$ and $\alpha $, we denote by $Y_{\mu , \alpha }$ a self-diffeomorphism on $N_g$ which is described as the result of pushing the regular neighborhood of $\mu $ once along $\alpha $.  We call $Y_{\mu , \alpha }$ a {\it Y-homeomorphism} (or {\it crosscap slide}). By Lemma~3.6 in \cite{Szepietowski1}, Y-homeomorphisms are in $\Gamma _2(N_g)$.
\begin{lem}\label{pushing2}
Suppose that $\gamma $ is a one-sided simple loop on $N_{g-1}$ based at $x_0$ such that $\gamma $ and $\partial D$ intersect at antipodal points of $\partial D$. Then we have
\[
\psi (\gamma )=Y_{\mu ,\widetilde {\gamma }},
\]
where  $\widetilde {\gamma }$ is a image of $\gamma $ to $N_g$ by a blowup and $\mu $ is a center line of the crosscap obtained from the regular neighborhood of $\partial D$ in $N_{g-1}$ by the blowup of $N_{g-1}$ on $D$ (See Figure~\ref{crosscap_def_y}).
\end{lem}

\begin{figure}[h]
\includegraphics[scale=0.59]{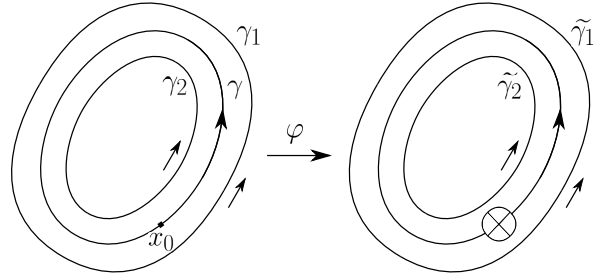}
\caption{A crosscap pushing map along two-sided simple loop $\gamma$.}\label{crosscap_def_twist}
\end{figure}

\begin{figure}[h]
\includegraphics[scale=0.59]{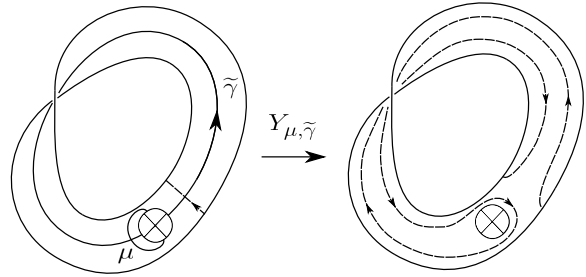}
\caption{A crosscap pushing map along one-sided simple loop $\gamma$ (Y-homeomorphism $Y_{\mu ,\widetilde {\gamma }}$).}\label{crosscap_def_y}
\end{figure}

Remark that the image of a crosscap pushing map is contained in $\Gamma _2(N_g)$. By Lemma~\ref{pushing1}, if $\gamma $ is a two-sided simple loop on $N_g$, then $\psi (\gamma )$ is an element of $\mathcal{T}_2(N_g)$. We remark that Y-homeomorphisms are not in $\mathcal{T}(N_g)$ (See~\cite{Lickorish2}).  

\subsection{Notation of the surface $N_g$}\label{surf_notation}

Let $e_i:D_i^\prime \hookrightarrow \Sigma _0$ for $i=1$, $2, \dots $, $g$ be smooth embeddings of unit disks $D_i^\prime \subset \mathbb C$ to a 2-sphere $\Sigma _0$ such that $D_i:=e_i(D_i^\prime )$ and $D_j$ are disjoint for distinct $1\leq i,j\leq g$, and let $x_i\in \Sigma _0$ for $i=1$, $2, \dots $, $g$ be $g$ points of $\Sigma _0$ such that $x_i$ is contained in the interior of $D_i$ as the left-hand side of Figure~\ref{nonorisurf}. Then $N_g$ is diffeomorphic to the surface obtained from $\Sigma _0$
by the blowups on $D_1,\dots ,D_g$. We describe the identification of $\partial D_i$ by the x-mark as the right-hand side of Figure~\ref{nonorisurf}. 
We call 
the crosscap which is obtained from the regular neighborhood of $\partial D_i$ in $\Sigma _0$ by the blowup of $\Sigma _0$ on $D_i$ the {\it $i$-th crosscap}.

We denote by $N_{g-1}^{(k)}$ the surface obtained from $\Sigma _0$
by the blowups on $D_i$ for every $i\not=k$. $N_{g-1}^{(k)}$ is diffeomorphic to $N_{g-1}$. Let $x_{k;i}$ be a simple loop on $N_g$ based at $x_k$ for $i\not=k$ as Figure~\ref{loop_x_ki}. Then the fundamental group $\pi _1(N_{g-1}^{(k)})=\pi _1(N_{g-1}^{(k)},x_k)$ of $N_{g-1}^{(k)}$ has the following presentation.
\[
\pi _1(N_{g-1}^{(k)})=\bigl< x_{k;1}, \dots , x_{k;k-1}, x_{k;k+1}, \dots , x_{k;g} \mid  x_{k;1}^2\dots x_{k;k-1}^2x_{k;k+1}^2\dots x_{k;g}^2=1\bigr>. 
\]

\begin{figure}[h]
\includegraphics[scale=0.80]{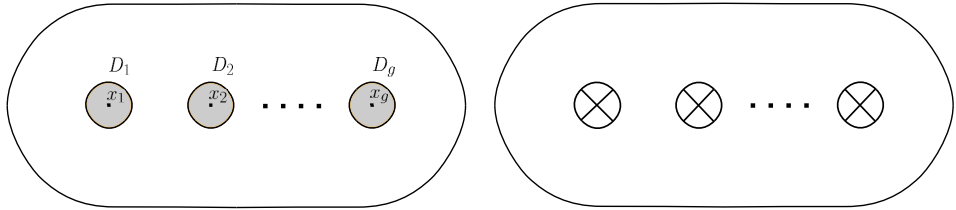}
\caption{The embedded disks $D_1$, $D_2$, $\dots$, $D_g$ on $\Sigma _0$ and the surface $N_g$.}\label{nonorisurf}
\end{figure}

\begin{figure}[h]
\includegraphics[scale=0.80]{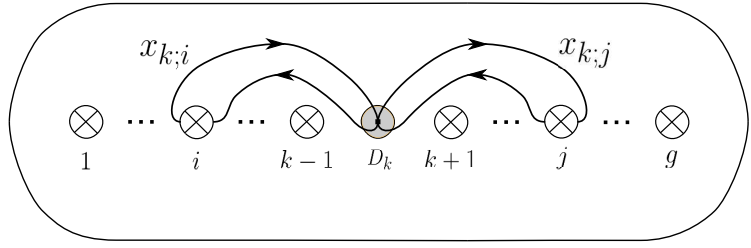}
\caption{The simple loop $x_{k;i}$ for $1\leq i\leq k-1$ and $x_{k;j}$ for $k+1\leq j\leq g$ on $N_{g-1}^{(k)}$ based at $x_k$.}\label{loop_x_ki}
\end{figure}

\subsection{Notations of mapping classes}

Let $\psi _k:\pi _1(N_{g-1}^{(k)})\rightarrow \mathcal{M}(N_g)$ be the crosscap pushing map obtained from 
the blowup of $N_{g-1}^{(k)}$ on $D_k$ and let $\pi _1(N_{g-1}^{(k)})^+$ be the subgroup of $\pi _1(N_{g-1}^{(k)})$ generated by two-sided simple loops on $N_{g-1}^{(k)}$ based at $x_k$. By Lemma~\ref{pushing1}, we have $\psi _k(\pi _1(N_{g-1}^{(k)})^+)\subset \mathcal{T}_2(N_g)$. We define non-separating two-sided simple loops $\alpha _{k;i,j}$ and $\beta _{k;i,j}$ on $N_{g-1}^{(k)}$ based at $x_k$ as in Figure~\ref{loops_alpha_beta} for distinct $1\leq i<j\leq g$ and $1\leq k\leq g$. We also define $\alpha _{k;j,i}:=\alpha _{k;i,j}$ and $\beta _{k;j,i}:=\beta _{k;i,j}$ for distinct $1\leq i<j\leq g$ and $1\leq k\leq g$. We have the following equations:
\begin{align*}
\alpha _{k;i,j}=x_{k;i}x_{k;j}  & \hspace{1cm} \text{for } i<j<k \text{ or }\ j<k<i \text{ or }\ k<i<j, \\
\beta _{k;i,j} =x_{k;j}x_{k;i}  & \hspace{1cm} \text{for } i<j<k \text{ or }\ j<k<i \text{ or }\ k<i<j.  
\end{align*}
Denote the crosscap pushing maps $a_{k;i,j}:=\psi _k(\alpha _{k;i,j})$ and $b_{k;i,j}:=\psi _k(\beta _{k;i,j})$. Remark that $a_{k;i,j}$ and $b_{k;i,j}$ are contained in the image of $\psi _k|_{\pi _1(N_{g-1}^{(k)})^+}$. Let $\eta $ be the self-diffeomorphism on $N_g$ which is the rotation of $N_g$ such that $\eta $ sends the $i$-th crosscap to the $(i+1)$-st crosscap for $1\leq i\leq g-1$ and the $g$-th crosscap to the $1$-st crosscap as Figure~\ref{rotation}. Then we have $a_{k;i,j}=\eta ^{k-1}a_{1;i-k+1,j-k+1}\eta ^{-(k-1)}$ and $b_{k;i,j}=\eta ^{k-1}b_{1;i-k+1,j-k+1}\eta ^{-(k-1)}$ for each distinct $1\leq i, j, k\leq g$.

\begin{figure}[h]
\includegraphics[scale=0.65]{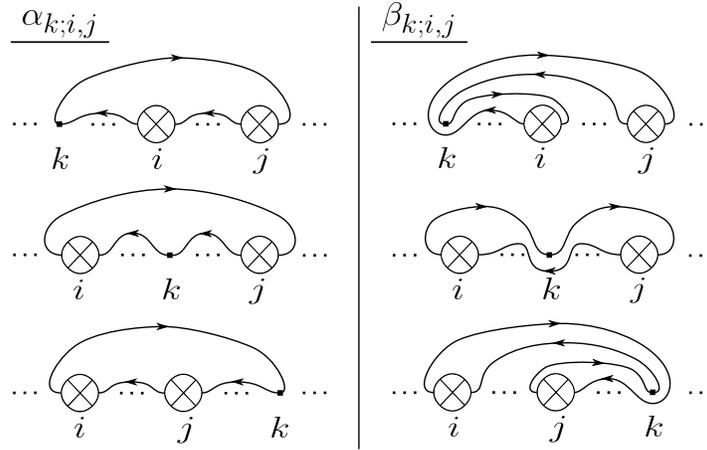}
\caption{Two-sided simple loops $\alpha _{k;i,j}$ and $\beta _{k;i,j}$ on $N_{g-1}^{(k)}$ based at $x_k$.}\label{loops_alpha_beta}
\end{figure}

\begin{figure}[h]
\includegraphics[scale=0.65]{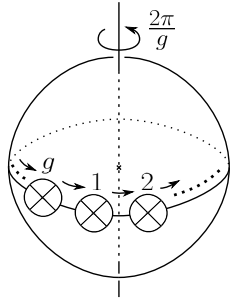}
\caption{The self-diffeomorphism $\eta $ on $N_g$.}\label{rotation}
\end{figure}

For distinct $i_1, i_2, \dots , i_n\in \{1, 2, \dots , g\}$, we define a simple closed curve $\alpha _{i_1,i_2,\dots ,i_n}$ on $N_g$ as in Figure~\ref{scc_gammai1i2in}. The arrow on the side of the simple closed curve $\alpha _{i_1,i_2,\dots ,i_n}$ in Figure~\ref{scc_gammai1i2in} indicates the direction of the Dehn twist $t_{\alpha _{i_1,i_2,\dots ,i_n}}$ along $\alpha _{i_1,i_2,\dots ,i_n}$ if $n$ is even. We set
the notations of Dehn twists and Y-homeomorphisms as follows:

\begin{eqnarray*}
T_{i,j} &:=& t_{\alpha _{i,j}}   \hspace{1cm} \text{for } 1\leq i<j\leq g, \\
T_{i,j,k,l} &:=& t_{\alpha _{i,j,k,l}}   \hspace{1cm} \text{for } g\geq 4  \text{ and }1\leq i<j<k<l\leq g,\\
Y_{i,j} &:=& Y_{\alpha _{i},\alpha _{i,j}}=\psi _i(x_{i;j})  \hspace{1cm} \text{for } \text{ distinct } 1\leq i, j\leq g.  
\end{eqnarray*}
Note that $T_{i,j}^2$ and $T_{i,j,k,l}^2$ are elements of $\mathcal{T}_2(N_g)$, $Y_{i,j}$ is an element of $\Gamma _2(N_g)$ but $Y_{i,j}$ is not an element of $\mathcal{T}_2(N_g)$. We remark that $a_{k;i,j}=b_{k;i,j}^{-1}=T_{i,j}^2$ for any distinct $i, j, k\in \{1, 2, 3\}$ when $g=3$.

\begin{figure}[h]
\includegraphics[scale=0.65]{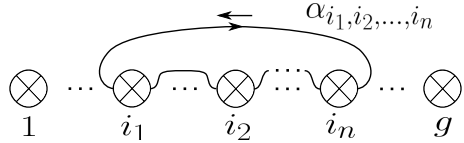}
\caption{The simple closed curve $\alpha _{i_1,i_2,\dots ,i_n}$ on $N_g$.}\label{scc_gammai1i2in}
\end{figure}


\section{Finite generating set for $\mathcal{T}_2(N_g)$}\label{section-finite-gen}

In this section, we prove the main theorem in this paper. The main theorem 
is as follows:

\begin{thm}\label{main-thm}
For $g\geq 3$, $\mathcal{T}_2(N_g)$ is generated by the following elements:
\begin{enumerate}
 \item[(i)] $a_{k;i,i+1}$, $b_{k;i,i+1}$, $a_{k;k-1,k+1}$, $b_{k;k-1,k+1}$ \ for $1\leq k\leq g$, $1\leq i\leq g$ and $i\not=k-1, k$,
 \item[(ii)] $a_{1;2,4}$, $b_{k;1,4}$, $a_{l;1,3}$ \ for $k=2, 3$ and $4\leq l\leq g$ when $g$ is odd,
 \item[(iii)] $T_{1,j,k,l}^2$ \ for $2\leq j< k< l\leq g$ when $g\geq 4$,
\end{enumerate} 
where the indices are considered modulo $g$. 
\end{thm}

We remark that the number of generators in Theorem~\ref{main-thm} is $\frac{1}{6}(g^3+6g^2+5g-6)$ for $g\geq 4$ odd, $\frac{1}{6}(g^3+6g^2-g-6)$ for $g\geq 4$ even and $3$ for $g=3$.

\subsection{Finite generating set for $\pi _1(N_{g-1}^{(k)})^+$}
First, we have the following lemma:

\begin{lem}\label{index2_1}
For $g\geq 2$, $\pi _1(N_{g-1}^{(k)})^+$ is an index 2 subgroup of $\pi _1(N_{g-1}^{(k)})$.
\end{lem}
\begin{proof}
Note that $\pi _1(N_{g-1}^{(k)})$ is generated by $x_{k;1}, \dots , x_{k;k-1}, x_{k;k+1}, \dots , x_{k;g}$. If $g=2$, $\pi _1(N_{g-1}^{(k)})$ is isomorphic to $\mathbb Z_2$ which is generated by a one-sided simple loop. Hence $\pi _1(N_{g-1}^{(k)})^+$ is trivial and we obtain this lemma when $g=2$.

We assume that $g\geq 3$. For $i\not=k$, we have
\[
x_{k;i}=x_{k;k-1}^{-1}\cdot x_{k;k-1}x_{k;i}.
\]
Since $x_{k;k-1}x_{k;i}=\beta _{k;i,k-1}\in \pi _1(N_{g-1}^{(k)})^+$, the equivalence classes of $x_{k;i}$ and $x_{k;k-1}^{-1}$ in $\pi _1(N_{g-1}^{(k)})/\pi _1(N_{g-1}^{(k)})^+$ are the same. We also have
\[
x_{k;k-1}=x_{k;k-1}^{-1}\cdot x_{k;k-1}^2.
\]
Since $x_{k;k-1}^2\in \pi _1(N_{g-1}^{(k)})^+$, the equivalence classes of $x_{k;k-1}$ and $x_{k;k-1}^{-1}$ in $\pi _1(N_{g-1}^{(k)})/\pi _1(N_{g-1}^{(k)})^+$ are the same. Thus $\pi _1(N_{g-1}^{(k)})/\pi _1(N_{g-1}^{(k)})^+$ is generated by the equivalence class $[x_{k;k-1}]$ whose order is 2 and we have completed the proof of Lemma~\ref{index2_1}.
\end{proof}

$N_{g-1}^{(k)}$ is diffeomorphic to the surface on the left-hand side (resp. right-hand side) of Figure~\ref{nonorisurf4} when $g-1=2h+1$ (resp. $g-1=2h+2$). We take a diffeomorphism which sends $x_{k;i}$ for $i\not=k$ and $x_k$ as in Figure~\ref{loop_x_ki} to $x_{k;i}$ for $i\not=k$ and $x_k$ as in Figure~\ref{nonorisurf4} and identify $N_{g-1}^{(k)}$ with the surface in Figure~\ref{nonorisurf4} by the diffeomorphism. Denote by $p_k:\widetilde{N_{g-1}^{(k)}}\twoheadrightarrow N_{g-1}^{(k)}$ the orientation double covering of $N_{g-1}^{(k)}$ as in Figure~\ref{orientation_cov1}. Then $H_k:=(p_k)_\ast (\pi _1(\widetilde{N_{g-1}^{(k)}}))$ is an index 2 subgroup of $\pi _1(N_{g-1}^{(k)})$. Note that when $g-1=2h+1$, $\pi _1(\widetilde{N_{g-1}^{(k)}})$ is generated by $y_{k;i}$ for $1\leq i\leq 4h$, and when $g-1=2h+2$, $\pi _1(\widetilde{N_{g-1}^{(k)}})$ is generated by $y_{k;i}$ for $1\leq i\leq 4h+2$, where $y_{k;i}$ is two-sided simple loops on $\widetilde{N_{g-1}^{(k)}}$ based at the lift $\widetilde{x_k}$ of $x_k$ as in Figure~\ref{orientation_cov1}. We have the following Lemma.

\begin{figure}[h]
\includegraphics[scale=1.2]{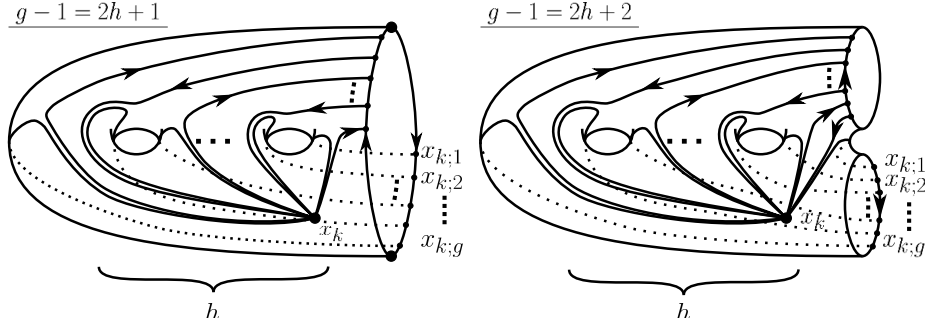}
\caption{$N_{g-1}^{(k)}$ is diffeomorphic to the surface on the left-hand side (resp. right-hand side) of the figure when $g-1=2h+1$ (resp. $g-1=2h+2$). We regard the above surface on the left-hand side as the surface identified antipodal points of the boundary component, and the above surface on the right-hand side as the surface attached their boundary components along the orientation of the boundary.}\label{nonorisurf4}
\end{figure}

\begin{lem}\label{index2_2}
For $g-1\geq 1$ and $1\leq k\leq g$,
\[
H_k=\pi _1(N_{g-1}^{(k)})^+.
\]
\end{lem}
\begin{proof}
Note that $\pi _1(N_{g-1}^{(k)})^+$ is an index 2 subgroup of $\pi _1(N_{g-1}^{(k)})$ by Lemma~\ref{index2_1}. It is sufficient for proof of Lemma~\ref{index2_2} to prove $H_k\subset \pi _1(N_{g-1}^{(k)})^+$ because the index of $H_k$ in $\pi _1(N_{g-1}^{(k)})$ is 
\begin{align*}
2&=[\pi _1(N_{g-1}^{(k)}):H_k]=[\pi _1(N_{g-1}^{(k)}):\pi _1(N_{g-1}^{(k)})^+][\pi _1(N_{g-1}^{(k)})^+:H_k]\\
&=2\cdot [\pi _1(N_{g-1}^{(k)})^+:H_k]
\end{align*}
if $H_k\subset \pi _1(N_{g-1}^{(k)})^+$.

We define subsets of $\pi _1(N_{g-1}^{(k)})^+$ as follows:
\begin{eqnarray*}
A&:=&\{ x_{k;j+1}x_{k;j},x_{k;k+1}x_{k;k-1} \mid 1\leq j\leq g-1,\ j\not=k-1,k\},\\
B&:=&\{ x_{k;j}x_{k;j+1},x_{k;k-1}x_{k;k+1} \mid 1\leq j\leq g-1,\ j\not=k-1,k\},\\
C&:=& \left\{ \begin{array}{ll}
 \{ x_{k;1}^2 \}&\text{if} \ k\not= 1,   \\
 \{ x_{k;2}^2 \}&\text{if} \ k=1.
 \end{array} \right.
\end{eqnarray*} 
$\pi _1(\widetilde{N_{g-1}^{(k)}})$ is generated by $y_{k;i}$. For $i\leq 2h$ when $g-1=2h+1$ (resp. $i\leq 2h+1$ when $g-1=2h+2$), we can check that
\begin{eqnarray*}
(p_k)_\ast (y_{k;i})&=&\left\{ \begin{array}{ll}
 x_{k;\rho (i+1)}x_{k;\rho (i)} &\text{if} \ 2\leq i\leq 2h\text{ and } g-1=2h+1,\\
 x_{k;g}x_{k;g-1} &\text{if} \ i=1\text{ and } g-1=2h+1,\\ 
 x_{k;\rho (i+1)}x_{k;\rho (i)} &\text{if} \ 2\leq i\leq 2h+1\text{ and } g-1=2h+2,\\
 x_{k;g}x_{k;g-1} &\text{if} \ i=1\text{ and } g-1=2h+2,  
 \end{array} \right. \\
\end{eqnarray*}
and $(p_k)_\ast (y_{k;i})$ is an element of $A$, where $\rho $ is the order reversing bijection from $\{ 1,2,\dots ,2h\}$ (resp. $\{ 1,2,\dots ,2h+1\}$) to $\{ 1,2,\dots ,g-1\} -\{ k\}$. Since if $g-1=2h+1$ and $g-1=2h^\prime +2$, we have
\begin{eqnarray*}
(p_k)_\ast (y_{k;2h+1})&=&\left\{ \begin{array}{ll}
 x_{k;1}^2 &\text{if} \ k\not= 1,\\
 x_{k;2}^2 &\text{if} \ k=1, 
 \end{array} \right. \\
(p_k)_\ast (y_{k;2h^\prime +2})&=&\left\{ \begin{array}{ll}
 x_{k;1}^2 &\text{if} \ k\not= 1,\\
 x_{k;2}^2 &\text{if} \ k=1, 
 \end{array} \right. \\
\end{eqnarray*}
$(p_k)_\ast (y_{k;2h+1})$ and $(p_k)_\ast (y_{k;2h^\prime +2})$ are elements of $C$ respectively (See Figure~\ref{loopx_k1^2}). Finally, for $i\geq 2h+2$ when $g-1=2h+1$ (resp. $i\geq 2h+3$ when $g=2h+2$), we can also check that
\begin{eqnarray*}
(p_k)_\ast (y_{k;i})&=&\left\{ \begin{array}{ll}
 x_{k;\rho ^\prime (i)}x_{k;\rho ^\prime (i+1)} &\text{if} \ 2h+2\leq i\leq 4h-1\text{ and } g-1=2h+1,\\
 x_{k;g-1}x_{k;g} &\text{if} \ i=4h\text{ and } g-1=2h+1,\\ 
 x_{k;\rho ^\prime (i)}x_{k;\rho ^\prime (i+1)} &\text{if} \ 2h+3\leq i\leq 4h+1\text{ and } g-1=2h+2,\\
 x_{k;g-1}x_{k;g} &\text{if} \ i=4h+2\text{ and } g-1=2h+2,  
 \end{array} \right. \\
\end{eqnarray*}
and $(p_k)_\ast (y_{k;i})=x_{k;\rho ^\prime (i)}x_{k;\rho ^\prime (i+1)}$ is an element of $B$, where $\rho ^\prime $ is the order preserving bijection from $\{ 2h+2,2h+3,\dots ,4h\}$ (resp. $\{ 2h+3,2h+4,\dots ,4h+2\}$) to $\{ 1,2,\dots ,g-1\} -\{ k\}$ (See Figure~\ref{loopx_kjx_kj+1}). We obtain this lemma.

\end{proof}

\begin{figure}[h]
\includegraphics[scale=1.4]{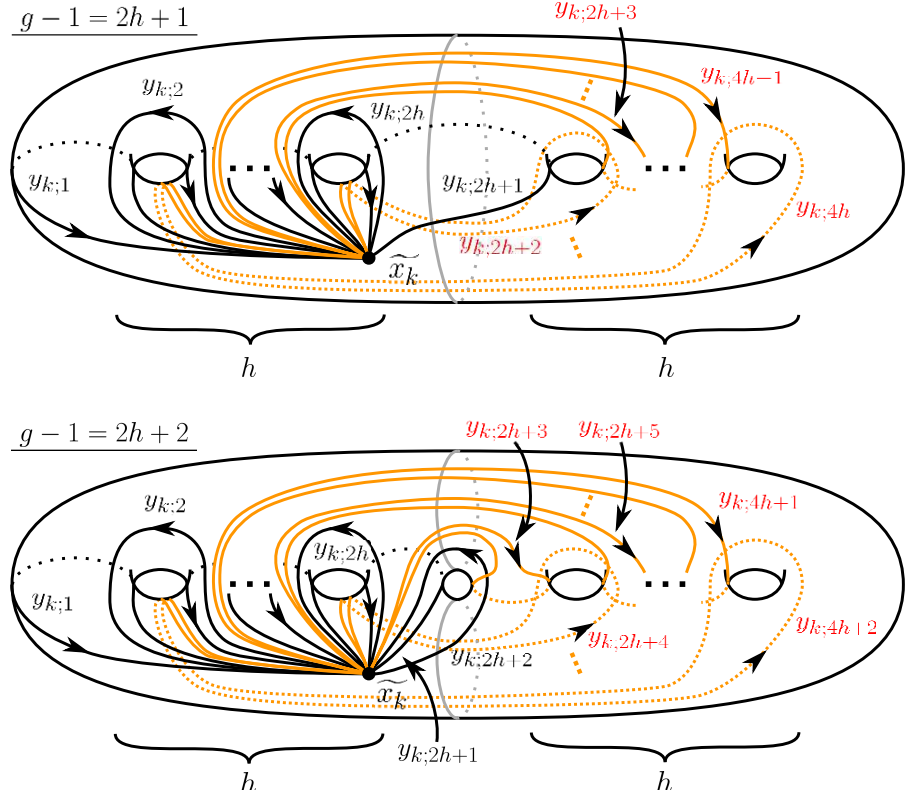}
\caption{The total space $\widetilde{N_{g-1}^{(k)}}$ of the orientation double covering $p_k$ of $N_{g-1}^{(k)}$ and two-sided simple loops $y_{k;i}$ on $\widetilde{N_{g-1}^{(k)}}$ based at $\widetilde{x_k}$.}\label{orientation_cov1}
\end{figure}

\begin{figure}[h]
\includegraphics[scale=1.0]{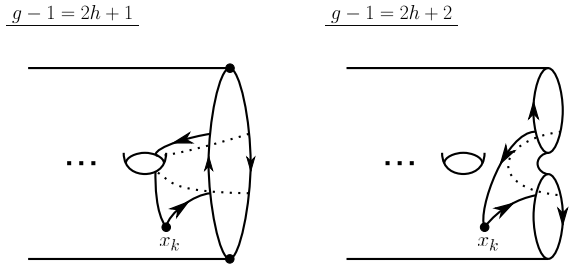}
\caption{The representative of $x_{k;1}^2$ when $k\not=1$ or $x_{k;2}^2$ when $k=1$.}\label{loopx_k1^2}
\end{figure}

\begin{figure}[h]
\includegraphics[scale=1.0]{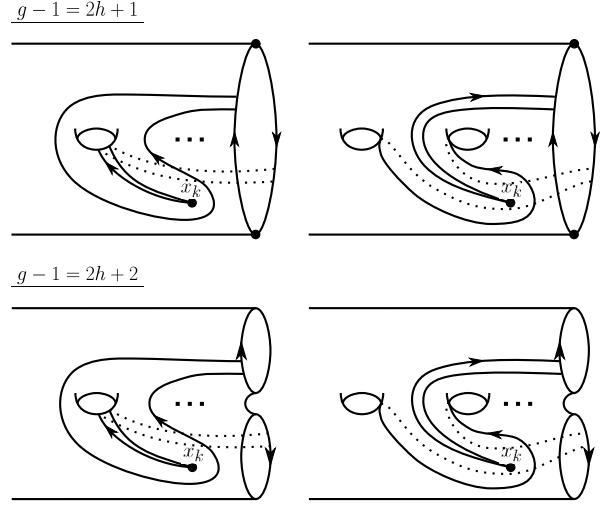}
\caption{The representative of $x_{k;j}x_{k;j+1}$ and $x_{k;k-1}x_{k;k+1}$ for $j\not=k-1,k$.}\label{loopx_kjx_kj+1}
\end{figure}

By the proof of Lemma~\ref{index2_2}, we have the following proposition.

\begin{prop}\label{gen1_pi_1plus}
For $g\geq 2$, $\pi _1(N_{g-1}^{(k)})^+$ is generated by the following elements:
\begin{enumerate}
 \item[(1)] $x_{k;i+1}x_{k;i}$, $x_{k;i}x_{k;i+1}$, $x_{k;k+1}x_{k;k-1}$, $x_{k;k-1}x_{k;k+1}$\ for $1\leq i\leq g-1$ and $i\not=k-1,k$,
 \item[(2)] $x_{k;2}^2$ \ when $k=1$,
 \item[(3)] $x_{k;1}^2$ \ when $2\leq k\leq g$.
\end{enumerate}
\end{prop}

We remark that $x_{k;i+1}x_{k;i}=\beta _{k;i,i+1}$, $x_{k;i}x_{k;i+1}=\alpha _{k;i,i+1}$, $x_{k;k+1}x_{k;k-1}=\alpha _{k;k-1,k+1}$, $x_{k;k-1}x_{k;k+1}=\beta _{k;k-1,k+1}$ and $Y_{i,j}^2=Y_{j,i}^2$. Let $G$ be the subgroup of $\mathcal{T}_2(N_g)$ generated by $\cup _{k=1}^g\psi _k(\pi _1(N_{g-1}^{(k)})^+)$. The next corollary follows from Proposition~\ref{gen1_pi_1plus} immediately.

\begin{cor}\label{cor2}
For $g\geq 2$, $G$ is generated by the following elements:
\begin{enumerate}
 \item[(i)] $a_{k;i,i+1}$, $b_{k;i,i+1}$, $a_{k;k-1,k+1}$, $b_{k;k-1,k+1}$ \ for $1\leq k\leq g$, $1\leq i\leq g-1$ and $i\not=k-1, k$,
 \item[(ii)] $Y_{1,j}^2$  \ when $2\leq j\leq g$,
\end{enumerate} 
where the indices are considered modulo $g$.
\end{cor}

The simple loop $x_{k;1}^2$ and $x_{k;2}^2$ are separating loops. By the next proposition, $\pi _1(N_{g-1}^{(k)})^+$ is generated by finitely many two-sided non-separating simple loops.   

\begin{prop}\label{gen2_pi_1plus}
For $g\geq 2$, $\pi _1(N_{g-1}^{(k)})^+$ is generated by the following elements:
\begin{enumerate}
 \item[(1)] $x_{k;i+1}x_{k;i}$, $x_{k;i}x_{k;i+1}$, $x_{k;k+1}x_{k;k-1}$, $x_{k;k-1}x_{k;k+1}$\ for $1\leq i\leq g$ and $i\not=k-1,k$,
 \item[(2)] $x_{k;2}x_{k;4}$ \ when $k=1$ and $g-1$ is even, 
 \item[(3)] $x_{k;1}x_{k;4}$ \ when $k=2,3$ and $g-1$ is even,
 \item[(4)] $x_{k;1}x_{k;3}$ \ when $4\leq k\leq g$ and $g-1$ is even,
\end{enumerate} 
where the indices are considered modulo $g$. 
\end{prop}

\begin{proof}
When $g-1$ is odd, since we have
\[
x_{1;2}^2=x_{1;2}x_{1;3}\cdot x_{1;3}^{-1}x_{1;4}^{-1}\cdot x_{1;4}x_{1;5}\cdot \cdots \cdot x_{1;g-1}^{-1}x_{1;g}^{-1}\cdot x_{1;g}x_{1;2}
\]
and
\[
x_{k;1}^2=x_{k;1}x_{k;2}\cdot x_{k;2}^{-1}x_{k;3}^{-1}\cdot x_{k;3}x_{k;4}\cdot \cdots \cdot x_{k;g-1}^{-1}x_{k;g}^{-1}\cdot x_{k;g}x_{k;1}
\]
for $2\leq k\leq g$, this proposition is clear.

When $g-1$ is even, we use the relation 
\[
x_{k;1}^2\dots x_{k;k-1}^2x_{k;k+1}^2\dots x_{k;g}^2=1.
\]
For $k=1$, we have
\[
x_{1;2}^2=x_{1;2}\underline{x_{1;4}\cdot x_{1;4}x_{1;5}\cdot \dots \cdot x_{1;g}x_{1;2}\cdot x_{1;2}x_{1;3}\cdot x_{1;3}}x_{1;2}.
\] 
By a similar argument, we also have following equations:
\[
x_{k;1}^2= \left\{ \begin{array}{ll}
  x_{2;1}^2  =x_{2;1}\underline{x_{2;4}\cdot x_{2;4}x_{2;5}\cdot \dots \cdot x_{2;g}x_{2;1}\cdot x_{2;1}x_{2;3}\cdot x_{2;3}}x_{2;1} &\hspace{-0.15cm}\text{if} \ k=2,\\
  x_{3;1}^2  =x_{3;1}\underline{x_{3;4}\cdot x_{3;4}x_{3;5}\cdot \dots \cdot x_{3;g}x_{3;1}\cdot x_{3;1}x_{3;2}\cdot x_{3;2}}x_{3;1} &\hspace{-0.15cm}\text{if} \ k=3,\\
  x_{k;1}^2  =x_{k;1}\underline{x_{k;3}\cdot x_{k;3}x_{k;4}\cdot \dots \cdot x_{k;k-1}x_{k;k+1}\cdot \cdots } & \\
 \hspace{1.5cm}\underline{\cdot x_{k;g}x_{k;1}\cdot x_{k;1}x_{k;2}\cdot x_{k;2}}x_{k;1} &\hspace{-0.15cm}\text{if} \ 4\leq k\leq g. 
 \end{array} \right.
\]
We obtain this proposition.
\end{proof}

We remark that $x_{k;1}x_{k;g}=\beta _{k;1,g}$, $x_{k;g}x_{k;1}=\alpha _{k;1,g}$, $x_{1;2}x_{1;4}=\alpha _{1;2,4}$, $x_{k;1}x_{k;4}=\beta _{k;1,4}$ for $k=2,3$ and $x_{k;1}x_{k;3}=\alpha _{k;1,3}$ for $4\leq k\leq g$. By the above remarks, we have the following corollary.

\begin{cor}\label{cor1}
For $g\geq 2$, $G$ is generated by the following elements:
\begin{enumerate}
 \item[(i)] $a_{k;i,i+1}$, $b_{k;i,i+1}$, $a_{k;k-1,k+1}$, $b_{k;k-1,k+1}$ \ for $1\leq k\leq g$, $1\leq i\leq g$ and $i\not=k-1, k$,
 \item[(ii)] $a_{1;2,4}$, $b_{k;1,4}$, $a_{l;1,3}$ \ for $k=2, 3$ and $4\leq l\leq g$ when $g$ is odd,
\end{enumerate} 
where the indices are considered modulo $g$. 
\end{cor}

\subsection{Proof of Main-Theorem}\label{proof_main}
First, we obtain a finite generating set for $\mathcal{T}_2(N_g)$ by the Reidemeister-Schreier method. We use the following minimal generating set for $\Gamma _2(N_g)$ given by Hirose and Sato~\cite{Hirose-Sato} when $g\geq 5$ and Szepietowski~\cite{Szepietowski2} when $g=3,4$ to apply the Reidemeister-Schreier method. See for instance \cite{Johnson} to recall the Reidemeister-Schreier method.

\begin{thm}\cite{Hirose-Sato,Szepietowski2}\label{minimal_gen}
For $g\geq 3$, $\Gamma _2(N_g)$ is generated by the following elements:
\begin{enumerate}
 \item[(1)] $Y_{i,j}$\ for $1\leq i\leq g-1$, $1\leq j\leq g$ and $i\not=j$,
 \item[(2)] $T_{1,j,k,l}^2$ \ for $2\leq j< k< l\leq g$ when $g\geq 4$.
\end{enumerate} 
\end{thm}

\begin{prop}\label{r-m}
For $g\geq 3$, $\mathcal{T}_2(N_g)$ is generated by the following elements:
\begin{enumerate}
 \item[(1)] $Y_{i,j}Y_{1,2}$, $Y_{i,j}^2$\ for $1\leq i\leq g-1$, $1\leq j\leq g$ and $i\not=j$,
 \item[(2)] $Y_{1,2}^{-1}T_{1,j,k,l}^2Y_{1,2}$, $T_{1,j,k,l}^2$ \ for $2\leq j< k< l\leq g$ when $g\geq 4$.
\end{enumerate} 
\end{prop}  

\begin{proof}
Note that $\mathcal{T}_2(N_g)$ is the intersection of $\Gamma _2(N_g)$ and $\mathcal{T}(N_g)$. Hence we have the isomorphisms
\[
\Gamma _2(N_g)/(\Gamma _2(N_g)\cap \mathcal{T}(N_g))\cong (\Gamma _2(N_g)\mathcal{T}(N_g))/\mathcal{T}(N_g)\cong \mathbb Z_2[Y_{1,2}].
\]
We remark that $\Gamma _2(N_g)\mathcal{T}(N_g)=\mathcal{M}(N_g)$ and the last isomorphism is given by Lickorish~\cite{Lickorish2}. Thus $\mathcal{T}_2(N_g)$ is an index $2$ subgroup of $\Gamma _2(N_g)$. 

Set $U:=\{ Y_{1,2},1\}$ and $X$ as the generating set for $\Gamma _2(N_g)$ in Theorem~\ref{minimal_gen}, where $1$ means the identity element. Then $U$ is a Schreier transversal for $\mathcal{T}_2(N_g)$ in $\Gamma _2(N_g)$. For $x\in \Gamma _2(N_g)$, define $\overline{x}$ as the element of $U$ such that $[\overline{x}]=[x]$ in $\Gamma _2(N_g)/\mathcal{T}_2(N_g)$. By the Reidemeister-Schreier method, for $g\geq 4$, $\mathcal{T}_2(N_g)$ is generated by
\begin{align*}
B=&\{ \overline{wu}^{-1}wu \mid w\in X^\pm ,\ u\in U,\ wu\not\in U \}\\
=&\{ \overline{Y_{i,j}^{\pm 1}Y_{1,2}}^{-1}Y_{i,j}^{\pm 1}Y_{1,2}, \overline{Y_{i,j}^{\pm 1}}^{-1}Y_{i,j}^{\pm 1} \mid 1\leq i\leq g-1,\ 1\leq j\leq g,\ i\not=j\}\\
&\cup \{ \overline{T_{1,j,k,l}^{\pm 2}Y_{1,2}}^{-1}T_{1,j,k,l}^{\pm 2}Y_{1,2}, \overline{T_{1,j,k,l}^{\pm 2}}^{-1}T_{1,j,k,l}^{\pm 2} \mid 2\leq j< k< l\leq g\}\\
=&\{ Y_{i,j}^{\pm 1}Y_{1,2}, Y_{1,2}^{-1}Y_{i,j}^{\pm 1}\mid 1\leq i\leq g-1,\ 1\leq j\leq g,\ i\not=j\}\\
&\cup \{ Y_{1,2}^{-1}T_{1,j,k,l}^{\pm 2}Y_{1,2}, T_{1,j,k,l}^{\pm 2} \mid 2\leq j< k< l\leq g\},
\end{align*}
where $X^\pm:=X\cup \{ x^{-1}\mid x\in X\}$ and note that equivalence classes of Y-homeomorphisms in $\Gamma _2(N_g)/\mathcal{T}_2(N_g)$ is nontrivial. Since $Y_{1,2}^{-1}Y_{i,j}^{\pm 1}=(Y_{i,j}^{\mp 1}Y_{1,2})^{-1}$ and $Y_{i,j}^{-1}Y_{1,2}=Y_{i,j}^{-2}\cdot Y_{i,j}Y_{1,2}$, we have the following generating set for $\mathcal{T}_2(N_g)$:
\begin{align*}
B^\prime =&\{ Y_{i,j}Y_{1,2}, Y_{i,j}^{2}\mid 1\leq i\leq g-1,\ 1\leq j\leq g,\ i\not=j\}\\
&\cup \{ Y_{1,2}^{-1}T_{1,j,k,l}^2Y_{1,2}, T_{1,j,k,l}^2 \mid 2\leq j< k< l\leq g\}.
\end{align*}
By a similar discussion, $\mathcal{T}_2(N_3)$ is generated by
\begin{align*}
B^\prime =&\{ Y_{i,j}Y_{1,2}, Y_{i,j}^{2}\mid 1\leq i\leq 2,\ 1\leq j\leq 3,\ i\not=j\}.
\end{align*}

We obtain this proposition.
\end{proof}

Let $\mathcal{G}$ be the group generated by the elements of type (i), (ii) and (iii) in Theorem~\ref{main-thm}. Then $\mathcal{G}$ is a subgroup of $\mathcal{T}_2(N_g)$ clearly and it is sufficient for the proof of Theorem~\ref{main-thm} to prove $B^\prime \subset \mathcal{G}$, where $B^\prime $ is the generating set for $\mathcal{T}_2(N_g)$ in the proof of Proposition~\ref{r-m}. By Corollary~\ref{cor1}, we have $\psi _k(\pi _1(N_{g-1}^{(k)})^+)\subset \mathcal{G}$ for any $1\leq k\leq g$. Thus $Y_{i,j}^{2}=\psi _i(x_{i;j}^2)\in \psi _i(\pi _1(N_{g-1}^{(i)})^+)\subset \mathcal{G}$. We complete the proof of Theorem~\ref{main-thm} if $Y_{i,j}Y_{1,2}$ and $Y_{1,2}^{-1}T_{1,j,k,l}^2Y_{1,2}$ are in $\mathcal{G}$.

\begin{lem}
For $g\geq 4$, $Y_{1,2}^{-1}T_{1,j,k,l}^2Y_{1,2}\in \mathcal{G}$.
\end{lem}

\begin{proof}
Since $Y_{1,2}^{-1}T_{1,j,k,l}^2Y_{1,2}=t_{Y_{1,2}^{-1}(\alpha _{1,j,k,l})}$, $Y_{1,2}^{-1}T_{1,j,k,l}^2Y_{1,2}$ is a Dehn twist along the two-sided simple closed curve as in Figure~\ref{scc_y_gamma1jkl}. Then we have $a_{1;k,l}(\alpha _{1,2,k,l})=Y_{1,2}^{-1}(\alpha _{1,2,k,l})$ and $Y_{1,2}^{-2}a_{1;2,j}a_{1;k,l}(\alpha _{1,j,k,l})=Y_{1,2}^{-1}(\alpha _{1,j,k,l})$ for $3\leq j\leq g$ and the local orientation of the regular neighborhood of $a_{1;k,l}(\alpha _{1,2,k,l})$ (resp. $Y_{1,2}^{-2}a_{1;2,j}a_{1;k,l}(\alpha _{1,j,k,l})$) and $Y_{1,2}^{-1}(\alpha _{1,2,k,l})$ (resp. $Y_{1,2}^{-1}(\alpha _{1,j,k,l})$) are different. Therefore we have
\begin{align*}
Y_{1,2}^{-1}T_{1,2,k,l}^2Y_{1,2}=& a_{1;k,l}T_{1,2,k,l}^{-2}a_{1;k,l}^{-1},\\
Y_{1,2}^{-1}T_{1,j,k,l}^2Y_{1,2}=& Y_{1,2}^{-2}a_{1;2,j}a_{1;k,l}T_{1,j,k,l}^{-2}a_{1;k,l}^{-1}a_{1;2,j}^{-1}Y_{1,2}^2\ \text{for }3\leq j\leq g.
\end{align*}
By Corollary~\ref{cor1}, $a_{1;k,l},a_{1;2,j}\in \psi _1(\pi _1(N_{g-1}^{(1)})^+)\subset \mathcal{G}$. We obtain this lemma.
\end{proof}

\begin{figure}[h]
\includegraphics[scale=0.65]{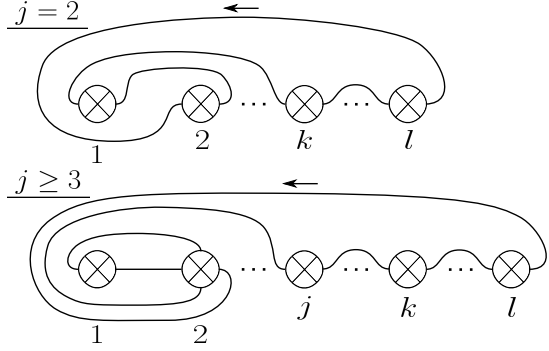}
\caption{The upper side of the figure is the simple closed curve $Y_{1,2}^{-1}(\alpha _{1,2,k,l})$ on $N_g$ and the lower side of the figure is the simple closed curve $Y_{1,2}^{-1}(\alpha _{1,j,k,l})$ on $N_g$ for $3\leq j\leq g$.}\label{scc_y_gamma1jkl}
\end{figure}

Szepietowski~\cite[Lemma~3.1]{Szepietowski1} showed that for any non-separating two-sided simple closed curve $\gamma $, $t_\gamma ^2$ is a product of two Y-homeomorphisms. In particular, we have the following lemma.

\begin{lem}[\cite{Szepietowski1}]\label{t2_yy} 
For distinct $1\leq i,j\leq g$, 
\[
Y_{j,i}^{-1}Y_{i,j}=Y_{j,i}Y_{i,j}^{-1}=\left\{ \begin{array}{ll}
 T_{i,j}^2 &\hspace{-0.15cm}\text{for} \ i<j,\\
 T_{i,j}^{-2} &\hspace{-0.15cm}\text{for} \ j<i. 
 \end{array} \right.
\]
\end{lem}

\begin{lem}\label{t2} 
For distinct $1\leq i,j\leq g$, $T_{i,j}^2\in \mathcal{G}$.
\end{lem}

\begin{proof}
We discuss by a similar argument in proof of Lemma~3.5 in \cite{Szepietowski2}. Let $\gamma _i$ be the two-sided simple loop on $N_{g-1}^{(i)}$ for $i=3,\dots ,g$ as in Figure~\ref{squere_crosscap1}. Then we have $T_{1,2}^2=\psi _g(\gamma _g)\cdots \psi _4(\gamma _4)\psi _3(\gamma _3)$ (see Figure~\ref{squere_crosscap1}). Since $\gamma _i\in \pi _1(N_{g-1}^{(i)})^+$, each $\psi _i(\gamma _i)$ is an element of $\mathcal{G}$ by Corollary~\ref{cor1}. Hence we have $T_{1,2}^2\in \mathcal{G}$. 

We denote by $\sigma _{i,j}$ the self-diffeomorphism on $N_g$ which is obtained by the transposition of the $i$-th crosscap and the $j$-th crosscap as in Figure~\ref{sigma_ij}. 
$\sigma _{i,j}$ is called the {\it crosscap transposition} (c.f.~\cite{Paris-Szepietowski}). 
For $1\leq i<j\leq g$, put $f_{i,j} \in \mathcal{M}(N_g)$ as follows:
\begin{align*}
f_{1,2}&:=1,\\
f_{1,j}&:=\sigma _{j-1,j}\cdots \sigma _{3,4}\sigma _{2,3}\ \text{ for }3\leq j\leq g,\\
f_{i,j}&:=\sigma _{i-1.i}\cdots \sigma _{2,3}\sigma _{1,2}f_{1,j}\ \text{ for }2\leq i<j\leq g.
\end{align*}
Then $T_{i,j}^2=f_{i,j}T_{1,2}^2f_{i,j}^{-1}=f_{i,j}\psi _g(\gamma _g)f_{i,j}^{-1}\cdots f_{i,j}\psi _4(\gamma _4)f_{i,j}^{-1}\cdot f_{i,j}\psi _3(\gamma _3)f_{i,j}^{-1}$.
Since the action of $\sigma _{i,j}$ on $N_g$ preserves the set of $i$-th crosscaps for $1\leq i\leq g$, 
$f_{i,j}\psi _k(\gamma _k)f_{i,j}^{-1}$ is an element of $\psi _{k^\prime }(\pi _1(N_{g-1}^{(k^\prime )}))$ for some $k^\prime $. By Corollary~\ref{cor1}, we have $f_{i,j}\psi _k(\gamma _k)f_{i,j}^{-1}\in \mathcal{G}$ and we obtain this lemma.
\end{proof}

\begin{figure}[h]
\includegraphics[scale=0.65]{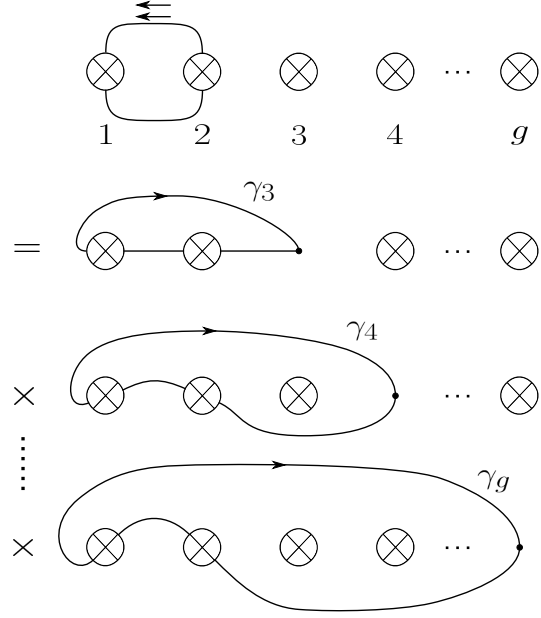}
\caption{$T_{1,2}^2$ is a product of crosscap pushing maps along $\gamma _3, \gamma _4,\dots ,\gamma _g$.}\label{squere_crosscap1}
\end{figure}

\begin{figure}[h]
\includegraphics[scale=0.65]{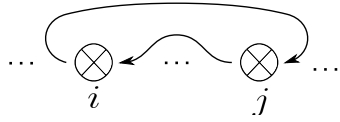}
\caption{The crosscap transposition $\sigma _{i,j}$.}\label{sigma_ij}
\end{figure}

Finally, by the following proposition, we complete the proof of Theorem~\ref{main-thm}.

\begin{prop}\label{y_kly_ij} 
For distinct $1\leq i,j,k,l\leq g$, $Y_{k,l}Y_{i,j}\in \mathcal{G}$.
\end{prop}

\begin{proof}
$Y_{k,l}Y_{i,j}$ is the following product of elements of $\mathcal{G}$.

(a) case $(k,l)=(i,j)$:
\[
Y_{k,l}Y_{i,j}=Y_{i,j}^2.
\] 
By Corollary~\ref{cor1}, the right-hand side is an element of $\mathcal{G}$. 

(b) case $(k,l)=(j,i)$:
\[
Y_{j,i}Y_{i,j}=Y_{j,i}^2\cdot Y_{j,i}^{-1}Y_{i,j}\stackrel{\text{Lem.~\ref{t2_yy} }}{=}(a)\cdot T_{i,j}^{\pm 2}.
\]
By Lemma~\ref{t2}, the right-hand side is an element of $\mathcal{G}$.

(c) case $k=i$ and $l\not=j$:
\[
Y_{i,l}Y_{i,j}=\psi _i(x_{i;l})\psi _i(x_{i;j})=\psi _i(\alpha _{i;j,l})=a_{i;j,l}\stackrel{\text{Cor.~\ref{cor1}}}{\in }\mathcal{G}.
\]

(d) case $k\not=i$ and $l=j$:
\[
Y_{k,j}Y_{i,j}=Y_{k,j}Y_{k,i}\cdot Y_{k,i}^{-1}Y_{i,k}^{-1}\cdot Y_{i,k}Y_{i,j}=(c)\cdot (b)\cdot (c) \in \mathcal{G}.
\]

(e) case $k=j$ and $l\not=i$:
\[
Y_{j,l}Y_{i,j}=Y_{j,l}Y_{j,i}\cdot Y_{j,i}^{-1}Y_{i,j}\stackrel{\text{Lem.~\ref{t2_yy} }}{=}(c)\cdot T_{i,j}^{\pm 2} \stackrel{\text{Lem.~\ref{t2}}}{\in } \mathcal{G}.
\]

(f) case $k\not=j$ and $l=i$:
\[
Y_{k,i}Y_{i,j}=Y_{k,i}Y_{i,k}^{-1}\cdot Y_{i,k}Y_{i,j}\stackrel{\text{Lem.~\ref{t2_yy} }}{=}T_{i,k}^{\pm 2}\cdot (c) \stackrel{\text{Lem.~\ref{t2}}}{\in }\mathcal{G}.
\]

(g) case $\{ k,l\}\cap \{ i,j\}$ is empty:
\[
Y_{k,l}Y_{i,j}=Y_{k,l}Y_{k,j}\cdot Y_{k,j}^{-1}Y_{k,j}^{-1}\cdot Y_{k,j}Y_{i,j}=(c)\cdot (a)\cdot (d) \in \mathcal{G}.
\]
We have completed this proposition.
\end{proof}

By a similar discussion in Subsection~\ref{proof_main} and Corollary~\ref{cor2}, we obtain the following theorem.
\begin{thm}\label{main-thm2}
For $g\geq 3$, $\mathcal{T}_2(N_g)$ is generated by following elements:
\begin{enumerate}
 \item[(i)] $a_{k;i,i+1}$, $b_{k;i,i+1}$, $a_{k;k-1,k+1}$, $b_{k;k-1,k+1}$ \ for $1\leq k\leq g$, $1\leq i\leq g-1$ and $i\not=k-1, k$,
 \item[(ii)] $Y_{1,j}^2$ \ for $2\leq j\leq g$,
 \item[(iii)] $T_{1,j,k,l}^2$ \ for $2\leq j< k< l\leq g$ when $g\geq 4$,
\end{enumerate} 
where the indices are considered modulo $g$.
\end{thm}

Since the number of generators in Theorem~\ref{main-thm2} is $\frac{1}{6}(g^3+6g^2-7g-12)$ for $g\geq 4$ and $3$ for $g=3$, the number of generators in Theorem~\ref{main-thm2} is smaller than the number of generators in Theorem~\ref{main-thm}. On the other hand, by Theorem~\ref{first-homology}, the dimension of the first homology group $H_1(\mathcal{T}_2(N_g))$ of $\mathcal{T}_2(N_g)$ is $\binom{g}{3}+\binom{g}{2}-1=\frac{1}{6}(g^3-g-6)$ for $g\geq 4$. The difference of them is $g^2-g-1$. The authors do not know the minimal number of generators for $\mathcal{T}_2(N_g)$ when $g\geq 4$. 

\section{Normal generating set for $\mathcal{T}_2(N_g)$}\label{section-normal-gen}

The next lemma is a generalization of the argument in the proof of Lemma~3.5 in \cite{Szepietowski2}.

\begin{lem}\label{t^2_product_crosscap}
Let $\gamma $ be a non-separating two-sided simple closed curve on $N_g$ such that $N_g-\gamma $ is a non-orientable surface. Then $t_\gamma ^2$ is a product of crosscap pushing maps along two-sided non-separating simple loops such that their crosscap pushing maps are conjugate to $a_{1;2,3}$ in $\mathcal{M}(N_g)$.
\end{lem}

\begin{proof}[Proof of Theorem~\ref{normal-gen}]

By Theorem~\ref{main-thm}, $\mathcal{T}_2(N_g)$ is generated by (I) crosscap pushing maps along non-separating two-sided simple loops and (I\hspace{-0.03cm}I) $T_{1,j,k,l}^2$ for $2\leq j< k< l\leq g$. When $g=3$, $\mathcal{T}_2(N_g)$ is generated by $T_{1,2}^2$, $T_{1,3}^2$, $T_{2,3}^2$. 
Recall $T_{i,j}^2=a_{k;i,j}^{-1}$ when $g=3$.
Since $N_g-\alpha _{i,j}$ is non-orientable for $g\geq 3$, 
$a_{k;i,j}$ is conjugate to $a_{k^\prime ;i^\prime ,j^\prime }$ in $\mathcal{M}(N_g)$. 
Hence Theorem~\ref{normal-gen} is clear when $g=3$.

Assume $g\geq 4$. For a non-separating two-sided simple loop $c$ on $N_{g-1}^{(k)}$ based at $x_k$, by Lemma~\ref{pushing1}, there exist non-separating two-sided simple closed curves $c_1$ and $c_2$ such that $\psi (c)=t_{c_1}t_{c_2}^{-1}$, where $c_1$ and $c_2$ are 
images of boundary components of regular neighborhood of $c$ in $N_{g-1}^{(k)}$ to $N_g$ by a blowup. Then the surface obtained by cutting $N_g$ along $c_1$ and $c_2$ is diffeomorphic to a disjoint sum of $N_{g-3,2}$ and $N_{1,2}$. Thus mapping classes of type (I) is conjugate to $a_{1;2,3}$ in $\mathcal{M}(N_g)$. We obtain Theorem~\ref{normal-gen} for $g=4$.

Assume $g\geq 5$. 
Simple closed curves $\alpha _{i,j,k,l}$ satisfy the condition of Lemma~\ref{t^2_product_crosscap}. 
Therefore $T_{1,j,k,l}^2$ is a product of crosscap pushing maps along non-separating two-sided simple loops and such crosscap pushing maps are conjugate to $a_{1;2,3}$ in $\mathcal{M}(N_g)$. We have completed the proof of Theorem~\ref{normal-gen}. 
\end{proof}


\section{First homology group of $\mathcal{T}_2(N_g)$}\label{section-first-homology}


By the argument in the proof of Proposition~\ref{r-m}, for $g\geq 2$, we have the following exact sequence:
\begin{eqnarray}\label{exact1}
1\longrightarrow \mathcal{T}_2(N_g)\longrightarrow \Gamma _2(N_g)\longrightarrow \mathbb Z _2\longrightarrow 0, 
\end{eqnarray}
where $\mathbb Z _2$ is generated by the equivalence class of a Y-homeomorphism. 

The {\it level 2 principal congruence subgroup} $\Gamma _2(n)$ of $GL(n,\mathbb Z)$ is the kernel of the natural surjection $GL(n,\mathbb Z)\twoheadrightarrow GL(n,\mathbb Z_2)$. Szepietowski~\cite[Corollary~4.2]{Szepietowski2} showed that there exists an isomorphism $\theta :\Gamma _2(N_3)\rightarrow \Gamma _2(2)$ which is induced by the action of $\Gamma _2(N_3)$ on the free part of $H_1(N_3;\mathbb Z)$. Since the determinant of the action of a Dehn twist on the free part of $H_1(N_3;\mathbb Z)$ is $1$, we have the following commutative diagram of exact sequences:
\begin{eqnarray}\label{exact2}
\xymatrix{
1 \ar[r]^{} & \mathcal{T}_2(N_3) \ar[d]_{\theta |_{\mathcal{T}_2(N_3)}} \ar[r]^{} \ar@{}[dr]|\circlearrowleft & \Gamma _2(N_3) \ar[d]_{\theta } \ar[r]^{} \ar@{}[dr]|\circlearrowleft & \mathbb Z_2 \ar[d]^{} \ar[r] & 1\\
1 \ar[r]^{} & SL(2,\mathbb Z)[2] \ar[r]^{} & \Gamma _2(2) \ar[r]^{\rm{det}} & \mathbb Z_2 \ar[r] & 1,\\
}\
\end{eqnarray}
where $SL(n,\mathbb Z)[2]:=\Gamma _2(n)\cap SL(n,\mathbb Z)$ is the {\it level 2 principal congruence subgroup} of the integral special linear group $SL(n,\mathbb Z)$. By the commutative diagram~(\ref{exact2}), $\mathcal{T}_2(N_3)$ is isomorphic to $SL(2,\mathbb Z)[2]$.

\begin{proof}[Proof of Theorem~\ref{first-homology}]
For $g=3$, the first homology group $H_1(\mathcal{T}_2(N_3))$ is isomorphic to $H_1(SL(2,\mathbb Z)[2])$ by the commutative diagram~(\ref{exact2}). The restriction of the natural surjection from $SL(2,\mathbb Z)$ to the projective special linear group $PSL(2,\mathbb Z)$ to $SL(n,\mathbb Z)[2]$ gives the following commutative diagram of exact sequences: 
\begin{eqnarray}\label{exact3}
\xymatrix{
1 \ar[r]^{} & \mathbb Z_2[-E] \ar[d]_{\rm{id}} \ar[r]^{} \ar@{}[dr]|\circlearrowleft & SL(2,\mathbb Z) \ar[d]_{} \ar[r]^{} \ar@{}[dr]|\circlearrowleft & PSL(2,\mathbb Z) \ar[d]^{} \ar[r] & 1\\
1 \ar[r]^{} & \mathbb Z_2[-E] \ar[r]^{} & SL(2,\mathbb Z)[2] \ar[r]^{} & PSL(2,\mathbb Z)[2] \ar[r] & 1,\\
}\
\end{eqnarray}
where $E$ is the identity matrix and $PSL(n,\mathbb Z)[2]:=SL(n,\mathbb Z)/\{ \pm E\}$ is the level 2 principal congruence subgroup of $PSL(2,\mathbb Z)$. Since $PSL(2,\mathbb Z)[2]$ is isomorphic to the free group $F_2$ of rank 2 and $-E$ commutes with all matrices, the exact sequence in the lower row of Diagram~(\ref{exact3}) is split and $SL(2,\mathbb Z)[2]$ is isomorphic to $F_2\oplus \mathbb Z_2$. Thus $H_1(\mathcal{T}_2(N_3))$ is isomorphic to $\mathbb Z^2\oplus \mathbb Z_2$.

For $g\geq 2$, the exact sequence~(\ref{exact1}) induces the five term exact sequence between these groups:
\[
H_2(\Gamma _2(N_g))\longrightarrow H_2(\mathbb Z _2)\longrightarrow H_1(\mathcal{T}_2(N_g))_{\mathbb Z _2}\longrightarrow H_1(\Gamma _2(N_g))\longrightarrow H_1(\mathbb Z _2)\longrightarrow 0,
\]
where 
\[
H_1(\mathcal{T}_2(N_g))_{\mathbb Z _2}:=H_1(\mathcal{T}_2(N_g))/\bigl< fm-m\mid m\in H_1(\mathcal{T}_2(N_g)),\ f\in \mathbb Z _2\bigr>.
\]
For $m\in H_1(\mathcal{T}_2(N_g))$ and $f\in \mathbb Z _2$, $fm:=[f^\prime m^\prime {f^\prime}^{-1}]\in H_1(\mathcal{T}_2(N_g))$ for some representative $m^\prime \in \mathcal{T}_2(N_g)$ and $f^\prime \in \Gamma _2(N_g)$. Since $H_2(\mathbb Z _2)\cong H_2(\mathbb RP^\infty )=0$ and $H_1(\mathbb Z _2)\cong \mathbb Z _2$, we have the short exact sequence:

\[
0\longrightarrow H_1(\mathcal{T}_2(N_g))_{\mathbb Z _2}\longrightarrow H_1(\Gamma _2(N_g))\longrightarrow \mathbb Z _2\longrightarrow 0.
\]   
Since Hirose and Sato~\cite{Hirose-Sato} showed that $H_1(\Gamma _2(N_g))\cong \mathbb Z_2^{\binom{g}{3}+\binom{g}{2}}$, it is sufficient for the proof of Theorem~\ref{first-homology} when $g\geq 5$ to prove that the action of $\mathbb Z_2\cong \Gamma _2(N_g)/\mathcal{T}_2(N_g)$ on the set of the first homology classes of generators for $\mathcal{T}_2(N_g)$ is trivial.

By Theorem~\ref{normal-gen}, $\mathcal{T}_2(N_g)$ is generated by crosscap pushing maps along non-separating two-sided simple loops for $g\geq 5$. Let $\psi (\gamma )=t_{\gamma _1}t_{\gamma _2}^{-1}$ be a crosscap pushing map along a non-separating two-sided simple loop $\gamma $, where $\gamma _1$ and $\gamma _2$ are images of boundary components of the regular neighborhood of $\gamma $ in $N_{g-1}$ to $N_g$ by a blowup. The surface $S$ obtained by cutting $N_g$ along $\gamma _1$ and $\gamma _2$ is diffeomorphic to a disjoint sum of $N_{g-3,2}$ and $N_{1,2}$. Since $g-3\geq 5-3=2$, we can define a Y-homeomorphism $Y$ on the component of $S$. The Y-homeomorphism is not a product of Dehn twists. Hence $[Y]$ is the nontrivial element in $\mathbb Z_2$ and clearly $Y\psi (\gamma )Y^{-1}=\psi (\gamma )$ in $\Gamma _2(N_g)$, i.e. $[Y\psi (\gamma )Y^{-1}]=[\psi (\gamma )]$ in $H_1(\mathcal{T}_2(N_g))$. Therefore the action of $\mathbb Z_2$ on $H_1(\mathcal{T}_2(N_g))$ is trivial and we have completed the proof of Theorem~\ref{first-homology}. 

\end{proof}

\begin{rem}
When $g=3$, we have the exact sequence
\[
0\longrightarrow H_1(\mathcal{T}_2(N_3))_{\mathbb Z _2}\longrightarrow H_1(\Gamma _2(N_3))\longrightarrow \mathbb Z _2\longrightarrow 0
\] 
by the argument in the proof of Theorem~\ref{first-homology}. Since $H_1(\Gamma _2(N_3))\cong H_1(\Gamma _2(2))\cong \mathbb Z_2^4$, we showed that the action of $\mathbb Z_2\cong \Gamma _2(N_3)/\mathcal{T}_2(N_3)$ on $H_1(\mathcal{T}_2(N_3))$ is not trivial by Theorem~\ref{first-homology} when $g=3$.
\end{rem}
\par
{\bf Acknowledgement: } The authors would like to express their gratitude to Hisaaki Endo and Susumu Hirose, for his encouragement and helpful advices. The authors also wish to thank Masatoshi Sato for his comments and helpful advices. The second author was supported by JSPS KAKENHI Grant number 15J10066.

\end{document}